\documentclass[10pt,reqno]{amsart}
\usepackage{amsmath}
\usepackage{amssymb}
\usepackage{amsthm}

\newlength{\defbaselineskip}
\newcommand{\setlinespacing}[1]%
           {\setlength{\baselineskip}{#1 \defbaselineskip}}

\numberwithin{equation}{section}

\newtheorem{thm}{Theorem}[section]
\newtheorem{cor}[thm]{Corollary}

\newtheorem{prop}[thm]{Proposition}

\theoremstyle{definition}

\theoremstyle{remark}
\newtheorem{rem}[thm]{Remark}
\numberwithin{equation}{section}

\begin{document}

\title[Unique continuation]
{Unique continuation for the Schr\"odinger equation with gradient term}

\author{Youngwoo Koh and Ihyeok Seo}

\thanks{Y. Koh was supported by NRF Grant 2016R1D1A1B03932049 (Republic of Korea).
I. Seo was supported by the NRF grant funded by the Korea government(MSIP) (No. 2017R1C1B5017496).}

\address{Department of Mathematics Education, Kongju National University, Kongju 32588, Republic of Korea}
\email{ywkoh@kongju.ac.kr}

\address{Department of Mathematics, Sungkyunkwan University, Suwon 16419, Republic of Korea}
\email{ihseo@skku.edu}

\maketitle


\begin{abstract}
We obtain a unique continuation result for the differential inequality
$| (i\partial_t +\Delta)u | \leq |Vu| + | W\cdot\nabla u |$
by establishing $L^2$ Carleman estimates.
Here, $V$ is a scalar function and $W$ is a vector function, which may be time-dependent or time-independent.
As a consequence, we give a similar result for the magnetic Schr\"odinger equation.
\end{abstract}


\section{Introduction}

Given a partial differential equation or inequality in $\mathbb{R}^n$, we say that it has
the unique continuation property from a non-empty open subset $\Omega\subset\mathbb{R}^n$
if its solution cannot vanish in $\Omega$ without being identically zero.
Historically, such property was studied in connection with the uniqueness of the Cauchy problem.
The major method for studying the property is based on so-called Carleman estimates which are weighted a priori estimates
for the solution.
The original idea goes back to Carleman \cite{C}, who first introduced it to obtain the property for the differential inequality
$|\Delta u|\leq|V(x)u|$ with $V\in L^\infty(\mathbb{R}^2)$
concerning the stationary Schr\"odinger equation.
Since then, the method has played a central role in almost all subsequent developments either for unbounded potentials $V$ or
fractional equations (see \cite{JK,St,KRS,CS,CR,S4,S5} and references therein).
For the equation involving gradient term,
\begin{equation}\label{gra}
|\Delta u|\leq|V(x)u|+|W(x)\cdot\nabla u|,
\end{equation}
see \cite{W,W2,LW} and references therein.

Now it can be asked whether the property is shared by the differential inequality
\begin{equation*}
|i\partial_tu +\Delta u|\leq|Vu|
\end{equation*}
concerning the (time-dependent) Schr\"odinger equation
which describes how the wave function $u$ of a non-relativistic quantum mechanical system with a potential $V$ evolves over time.
It would be an interesting problem to prove the property for this equation
since such property can be viewed as one of the non-localization properties of the wave function
which are an important issue in quantum mechanics.
The unique continuation for this time-dependent case has been studied for decades from a half-space $\Omega$ in $\mathbb{R}^{n+1}$
when the potential $V$ is time-dependent (\cite{KS,LS,S}) or time-independent (\cite{S3,S6}).
These results were based on Carleman estimates of the form
\begin{equation}\label{carl}
    \big\| e^{\beta\phi(x,t)} f \big\|_{\mathcal{B}}
    \leq C\big\| e^{\beta\phi(x,t)} (i\partial_t + \Delta) f \big\|_{\mathcal{B}'},
\end{equation}
where $\beta$ is a real parameter, $\phi$ is a suitable weight function,
and $\mathcal{B},\mathcal{B}'$ are suitable Banach spaces of functions on $\mathbb{R}^{n+1}$.

The new point in this paper is that we allow the gradient term $\nabla u$ in the differential inequality
\begin{equation}\label{Sch}
|i\partial_tu +\Delta u|\leq|Vu|+|W\cdot\nabla u|
\end{equation}
in the spirit of \eqref{gra}.
Here, $V$ is a scalar function and $W$ is a vector function, which may be time-dependent or time-independent.
From the physical point of view, a motivation behind this form comes also from the following magnetic Schr\"odinger equation
\begin{equation}\label{Sch2}
i\partial_tu +\Delta_{\vec{A}} u=Vu
\end{equation}
which describes the behavior of quantum mechanical systems in the presence of magnetic field.
Here, $\vec{A}=(A_1,...,A_n)$ is a magnetic vector potential
and $\Delta_{\vec{A}}$ denotes the magnetic Laplacian defined by
\begin{align*}
\Delta_{\vec{A}}&=\sum_j(\partial_j+iA_j)^2\\
&=\Delta+2i\vec{A}\cdot\nabla+i\textrm{div}\vec{A}-|\vec{A}|^2.
\end{align*}
Replacing $V$ and $W$ in \eqref{Sch} with $-i\textrm{div}\vec{A}+|\vec{A}|^2+V$ and $-2i\vec{A}$, respectively,
we can reduce the unique continuation problem for the magnetic case
to the one for the form \eqref{Sch}.

To obtain unique continuation for \eqref{Sch}, we find suitable weights $\phi$ and Banach spaces $\mathcal{B},\mathcal{B}'$
to allow the gradient term  in the left-hand side of \eqref{carl}, as follows
(see Proposition \ref{prop} for details):
\begin{equation*}
    \beta\big\| e^{\beta\phi(x,t)} f \big\|_{\mathcal{B}}+\beta^{1/2}\big\| e^{\beta\phi(x,t)}|\nabla f|\big\|_{\mathcal{B}}
    \leq C\big\| e^{\beta\phi(x,t)} (i\partial_t + \Delta) f \big\|_{\mathcal{B}'}
\end{equation*}
where $\phi(x,t)=ct+|x|^2$ and $\mathcal{B}=\mathcal{B}'=L_{x,t}^2$.
Making use of this estimate, we obtain the following unique continuation theorem
which says that if the solution of \eqref{Sch} is supported on inside of a paraboloid in $\mathbb{R}^{n+1}$,
then it must vanish on all of $\mathbb{R}^{n+1}$.

\begin{thm}\label{thm}
Let $V\in L^\infty$ and $|W|\in L^\infty$.
Suppose that $u\in H^1_t \cap H^2_x$ is a solution of \eqref{Sch} which vanishes in outside of a paraboloid given by
\begin{equation}\label{para}
\{(x,t)\in\mathbb{R}^{n+1}: c(t-t_0)+|x-x_0|^2 >0\},
\end{equation}
where $c\in\mathbb{R}\setminus\{0\}$ and $(x_0,t_0)\in\mathbb{R}^{n+1}$.
Then $u$ is identically zero in $\mathbb{R}^{n+1}$.
Here, $H_{t}^{1}$ denotes the space of functions whose derivatives up to order $1$,
with respect to the time variable $t$, belong to $L^2$.
$H_{x}^{2}$ is similarly defined as $H_{t}^{1}$.
\end{thm}

\begin{rem}
The assumption $u\in H^1_t \cap H^2_x$ can be relaxed to
\begin{equation}\label{cond}
u,\partial_t u,\Delta u\in L_t^2(\mathbb{R};L_x^2(|x-x_0|^2\leq c(t_0-t)))
\end{equation}
because we are assuming $u=0$ on the set $\{(x,t)\in\mathbb{R}^{n+1}: c(t-t_0)+|x-x_0|^2 >0\}$.
Since $u(x,t)$ vanishes at infinity for each $t$, it follows from integration by parts that $\nabla u$ also satisfies \eqref{cond}.
\end{rem}

Regarding the condition \eqref{cond}, we give a remark that the solutions $u(x,t)=e^{it\Delta}u_0(x)$ to the free Schr\"odinger equation
with the initial data $u_0\in C_0^\infty(\mathbb{R}^n)$, $n>2$, satisfy \eqref{cond}.
Indeed, we consider the case $c<0$ without loss of generality. Then we may consider $t\geq t_0$ only.
Since $e^{it\Delta}u_0(x)\in L_x^2$ for each $t$, it is enough to show that
\begin{equation}\label{free}
\int_{M}^\infty\int_{|x-x_0|^2\leq c(t_0-t)}|e^{it\Delta}u_0(x)|^2dxdt<\infty
\end{equation}
for a sufficiently large $M>t_0$.
But, using the following well-known decay
$$\sup_{x\in\mathbb{R}^n}|e^{it\Delta}u_0(x)|\lesssim|t|^{-n/2}\|u_0\|_{L^1(\mathbb{R}^n)},$$
the integral in \eqref{free} is bounded by
\begin{align*}
C\|u_0\|_{L^1(\mathbb{R}^n)}^2\int_{M}^\infty (t-t_0)^{n/2}t^{-n}dt
&\leq C\|u_0\|_{L^1(\mathbb{R}^n)}^2\int_{M}^\infty t^{-n/2}dt\\
&\leq C\|u_0\|_{L^1(\mathbb{R}^n)}^2
\end{align*}
if $n>2$.
Since $\partial_te^{it\Delta}u_0=e^{it\Delta}i\Delta u_0$
and $\Delta e^{it\Delta}u_0=e^{it\Delta}\Delta u_0$,
the condition \eqref{cond} for these follows from the same argument with $\Delta u_0\in L^1$.

As mentioned above, the theorem directly implies the following result for the magnetic case \eqref{Sch2}.

\begin{cor}
Let $V\in L^\infty$, $|\vec{A}|\in L^\infty$ and $\textrm{div}\vec{A}\in L^\infty$.
If\, $u\in H^1_t \cap H^2_x$ is a solution of \eqref{Sch2} which vanishes in outside of a paraboloid given by \eqref{para},
then $u$ is identically zero in $\mathbb{R}^{n+1}$.
\end{cor}

There are related results for Schr\"odinger equations
which describe the behavior of the solutions at two different times which
ensure $u\equiv0$.
Since the Schr\"odinger equation is time reversible,
it seems natural to consider this type of unique continuation from the behavior at two different time moments.
Such results have been first obtained by various authors (\cite{Z,KPV,IK,EKPV,S2})
for Schr\"odinger equations of the form
\begin{equation}\label{equ}
(i\partial_t+\Delta)u=Vu+F(u,\overline{u}),
\end{equation}
where $V(x,t)$ is a time-dependent potential and $F$ is a nonlinear term.

For the $1-D$ cubic Schr\"odinger equations, i.e., $V\equiv0$, $F=\pm|u|^2u$, $n=1$ in \eqref{equ},
Zhang \cite{Z} showed that if $u=0$ in the same semi-line at two times $t_0,t_1$, then $u\equiv0$ in $\mathbb{R}\times[t_0,t_1]$.
His proof is based on the inverse scattering theory.
In the work of Kenig-Ponce-Vega \cite{KPV}, this result was completely extended to higher dimensions under the assumption that
$u=0$ in the complement of a cone with opening $<\pi$ at two times.
Key steps in their proof were energy estimates for the Fourier transform of the solution and the use of Isakov's results \cite{I} on local unique continuation.
The size of the set on which the solution vanishes at two times was improved by Ionescu and Kenig \cite{IK}
to the case of semispaces, i.e., cones with opening $=\pi$.
Latter, Escauriaza-Kenig-Ponce-Vega \cite{EKPV}\footnote{This approach was motivated by a deep relationship between the unique continuation and uncertainty principles for the Fourier transform.
As a consequence, several remarkable results have been later obtained in both cases linear Schr\"odinger equations with variable coefficients
and nonlinear ones (see, for example, \cite{BFGRV} and references therein). See also a good survey paper \cite{EKPV2} explaining several results in the subject.} showed that it suffices to assume that the solution decay sufficiently fast at two times to have unique continuation results.

These results were extended in \cite{IK2,DS} to the case where the nonlinear term $F$ in \eqref{equ} involves the gradient terms $\nabla u,\nabla\overline{u}$.
More precisely, it was shown in \cite{IK2} that if the solution vanishes in the complement of a ball at two times $t_0,t_1$, then $u\equiv0$ in $\mathbb{R}^n\times[t_0,t_1]$.
In the spirit of \cite{EKPV}, this support condition at two times was improved in \cite{DS} to the assumption
that the solution decay sufficiently fast.
By applying these works to each time interval $[n,n+1]$, $n\in\mathbb{Z}$,
particularly for the linear equation
\begin{equation}\label{er}
i\partial_tu +\Delta u=Vu+W\cdot\nabla u,
\end{equation}
one can show that if the solution to \eqref{er} vanishes in outside of a paraboloid given by \eqref{para},
then $u\equiv0$ under the assumption that
\begin{equation}\label{pott}
V\in B_x^{2,\infty}L_t^\infty(\mathbb{R}^{n+1}),\quad |W|\in B_x^{1,\infty}L_t^\infty(\mathbb{R}^{n+1}),
\end{equation}
where $B^{p,q}$, $1\leq p,q\leq\infty$, are Banach spaces with the properties that
$B^{p,p}=L^p$, $1\leq p\leq\infty$, and
$B^{p_1,q_1}\hookrightarrow B^{p_2,q_2}$
if $q_1\geq q_2$ and $p_1\leq p_2$. (See \cite{IK2,DS} for details.)
But in this paper, the assumption \eqref{pott} is improved because
$$B_x^{p,\infty}L_t^\infty(\mathbb{R}^{n+1})\hookrightarrow B_x^{\infty,\infty}L_t^\infty(\mathbb{R}^{n+1})=L^\infty(\mathbb{R}^{n+1})$$
for $p=1,2$.

\section{$L^2$ Carleman estimate}\label{sec2}

In this section we obtain the following $L^2$ Carleman estimate which is a key ingredient for the proof of Theorem \ref{thm} in the next section.

\begin{prop}\label{prop}
Let $\beta>0$ and $c\in\mathbb{R}$.
Then we have for $f\in C_0^\infty(\mathbb{R}^{n+1})$
    \begin{equation}\label{key_eq_2}
    \beta \big\| e^{\beta(ct+|x|^2)} f \big\|_{L^2_{x,t}}
        + \beta^{\frac{1}{2}} \big\| e^{\beta(ct +|x|^2)} |\nabla f|\big \|_{L^2_{x,t}}
    \leq \big\| e^{\beta(ct +|x|^2)} (i\partial_t + \Delta) f \big\|_{L^2_{x,t}}.
    \end{equation}
\end{prop}

\begin{proof}
To show \eqref{key_eq_2}, we first set
$f= e^{-\beta(ct+|x|^2)}g$
and note that
    $$
    \begin{aligned}
    \big\| e^{\beta(ct +|x|^2)}|\nabla f| \big\|_{L^2_{x,t}}
    &= \big\| e^{\beta(ct +|x|^2)} | e^{-\beta(ct +|x|^2)} (-2\beta gx + \nabla g) | \big\|_{L^2_{x,t}} \\
    &\leq 2\beta \big\| |x| g \big\|_{L^2_{x,t}} + \big\| |\nabla g|\big\|_{L^2_{x,t}}.
    \end{aligned}
    $$
Hence it is enough to show that
   \begin{equation}\label{key_eq_3}
    \begin{aligned}
   \big\| e^{\beta(ct +|x|^2)} (i\partial_t + \Delta) &e^{-\beta(ct +|x|^2)}g \big\|_{L^2_{x,t}}\\
&\geq \beta \| g \|_{L^2_{x,t}}+ 2\beta^{\frac{3}{2}} \big\| |x| g \big\|_{L^2_{x,t}}
        + \beta^{\frac{1}{2}} \| |\nabla g| \|_{L^2_{x,t}}
    \end{aligned}
    \end{equation}
for $g\in C_0^\infty(\mathbb{R}^{n+1})$.

By direct calculation, we see that
$$e^{\beta(ct +|x|^2)}(i\partial_t + \Delta) e^{-\beta(ct +|x|^2)} g
    = ( i\partial_t + \Delta
        + 4\beta^2|x|^2 - 2n\beta - 4\beta x\cdot\nabla - i\beta c) g.$$
Let $A:= i\partial_t + \Delta + 4\beta^2|x|^2$
and
$B:= 2n\beta + 4\beta x\cdot\nabla_x + i\beta c$,
such that $A^*=A$ and $B^*=-B$.
(Here $A^*$ and $B^*$ denote adjoint operators.)
Then we get
    $$
    \begin{aligned}
    \big\| e^{\beta(ct +|x|^2)} (i\partial_t + \Delta)&e^{-\beta(ct +|x|^2)}g\big\|_{L^2_{x,t}}^2\\
    &=\big\langle (A-B)g, (A-B)g \big\rangle_{L^2_{x,t}}\\
    &=\big\langle Ag, Ag \big\rangle_{L^2_{x,t}}
    +\big\langle Bg, Bg \big\rangle_{L^2_{x,t}}
    +\big\langle (BA-AB)g, g \big\rangle_{L^2_{x,t}}\\
    &\geq\big\langle (BA-AB)g, g \big\rangle_{L^2_{x,t}}.
    \end{aligned}
    $$
Since the only terms in $A$ and $B$ which are not commutative each other are $4\beta x\cdot\nabla$ and $(\Delta + 4\beta^2|x|^2)$,
we see
    $$
    \begin{aligned}
    BA-AB
    &= (4\beta x\cdot\nabla)(\Delta + 4\beta^2|x|^2) - (\Delta + 4\beta^2|x|^2)(4\beta x\cdot\nabla) \\
    &= 32 \beta^3 |x|^2 - 8\beta\Delta.
    \end{aligned}
    $$
Hence it follows that
    \begin{equation}\label{lem_BA-AB}
    \begin{aligned}
    \big\langle (BA-AB)g, g \big\rangle_{L^2_{x,t}}
    &= 32\beta^3 \big\langle |x|^2 g, g \big\rangle_{L^2_{x,t}}
        + 8\beta \big\langle -\Delta g, g \big\rangle_{L^2_{x,t}}\\
    &= 32 \beta^3 \big\| |x| g \big\|_{L^2_{x,t}}^2 + 8\beta \big\| |\nabla g| \big\|_{L^2_{x,t}}^2 .
    \end{aligned}
    \end{equation}

Next we notice that
$$-\sum_{i=1}^n\frac12x_i\frac{\partial}{\partial x_i}|g(x,t)|^2=-\textrm{Re}(\nabla g\cdot \overline{xg}).$$
By integrating this over $\mathbb{R}^{n+1}$ and then integrating by parts on the left-hand side, we see
\begin{align*}
\frac n2\|g\|_{L_{x,t}^2}^2&=-\textrm{Re}\bigg(\int_{\mathbb{R}}\int_{\mathbb{R}^n}\nabla g\cdot \overline{xg}dxdt\bigg)\\
&\leq\bigg|\int_{\mathbb{R}}\int_{\mathbb{R}^n}\nabla g\cdot \overline{xg}dxdt\bigg|\\
&\leq\big\||\nabla g|\big\|_{L_{x,t}^2}\big\||x|g\big\|_{L_{x,t}^2}.
\end{align*}
Hence, by combining this inequality\footnote{If we consider $L_x^2$ instead of $L_{x,t}^2$,
this inequality is indeed the Heisenberg uncertainty principle in $n$ dimensions
($\||\nabla g|\|_{L_x^2}=2\pi\||\xi|\widehat{g}\|_{L_{\xi}^2}$).} and \eqref{lem_BA-AB},
we conclude that
$$\big\langle (BA-AB)g, g \big\rangle_{L^2_{x,t}}
\geq (31\beta^3)\big\| |x| g \big\|_{L^2_{x,t}}^2 + (7\beta) \big\| |\nabla g| \big\|_{L^2_{x,t}}^2 +n\beta^2\|g\|_{L_{x,t}^2}^2,$$
since
$$2\beta^2\big\||\nabla g|\big\|_{L_{x,t}^2}\big\||x|g\big\|_{L_{x,t}^2}
\leq\beta^3\big\||\nabla g|\big\|_{L_{x,t}^2}^2+\beta\big\||x|g\big\|_{L_{x,t}^2}^2.
$$

This implies \eqref{key_eq_3} because of $\sqrt{3}(a^2+b^2+c^2)^{1/2}\geq|a|+|b|+|c|$.
Indeed,
\begin{align*}
\big\| e^{\beta(ct +|x|^2)} (i\partial_t + \Delta)&e^{-\beta(ct +|x|^2)}g\big\|_{L^2_{x,t}}\\
    &=\big\langle (A-B)g, (A-B)g \big\rangle_{L^2_{x,t}}^{1/2}\\
    &\geq \Big(31\beta^3\big\| |x| g \big\|_{L^2_{x,t}}^2 + 7\beta\big\| |\nabla g| \big\|_{L^2_{x,t}}^2 +n\beta^2\|g\|_{L_{x,t}^2}^2\Big)^{1/2}\\
    &\geq\frac1{\sqrt{3}}\big(\sqrt{31\beta^3} \big\| |x| g \big\|_{L^2_{x,t}} + \sqrt{7\beta} \big\| |\nabla g| \big\|_{L^2_{x,t}} + \sqrt{n\beta^2} \| g \|_{L^2_{x,t}}\big)\\
    &\geq \beta \| g \|_{L^2_{x,t}}+ 2\beta^{\frac{3}{2}} \big\| |x| g \big\|_{L^2_{x,t}}
        + \beta^{\frac{1}{2}} \| |\nabla g| \|_{L^2_{x,t}}.
\end{align*}

\end{proof}

\section{Proof of Theorem \ref{thm}}\label{sec3}

This section is devoted to proving Theorem \ref{thm} using Proposition \ref{prop}.

By translation we may first assume that $(x_0,t_0)=(0,0)$ so that
the solution $u$ vanishes in the paraboloid $\{(x,t)\in\mathbb{R}^{n+1}: ct+|x|^2 >0 \}$.
Now, from induction it suffices to show that $u=0$ in the following set
\begin{equation*}
S=\{(x,t)\in\mathbb{R}^{n+1}:-1<ct+|x|^2\leq0\}.
\end{equation*}
To show this, we make use of the Carleman estimate in Proposition \ref{prop}.

Let $\psi:\mathbb{R}^{n+1}\rightarrow[0,\infty)$ be a
smooth function such that $\text{supp}\,\psi\subset B(0,1)$ and
$$\int_{\mathbb{R}^{n+1}}\psi(x,t)dxdt=1.$$
Also, let $\phi:\mathbb{R}^{n+1}\rightarrow[0,1]$ be a smooth
function such that $\phi=1$ in $B(0,1)$ and $\phi=0$ in $\mathbb{R}^{n+1}\setminus B(0,2)$.
For $0<\varepsilon<1$ and $R\geq1$, we set
$\psi_\varepsilon(x,t)=\varepsilon^{-(n+1)}\psi(x/\varepsilon,t/\varepsilon)$
and $\phi_R(x,t)=\phi(x/R,t/R)$.

Now we apply the Carleman estimate \eqref{key_eq_2} to the following $C_0^\infty$ function
\begin{equation*}
v(x,t)=(u\ast\psi_\varepsilon)(x,t)\phi_R(x,t).
\end{equation*}
Then, we see that
 \begin{equation*}
    \beta \big\| e^{\beta(ct+|x|^2)} v \big\|_{L^2_{x,t}}
        + \beta^{\frac{1}{2}} \big\| e^{\beta(ct +|x|^2)} |\nabla v|\big \|_{L^2_{x,t}}
    \leq \big\| e^{\beta(ct +|x|^2)} (i\partial_t + \Delta) v \big\|_{L^2_{x,t}}.
    \end{equation*}
Note that
    $$
    \nabla v= ( \nabla u \ast \psi_{\epsilon}) \phi_R
    + ( u \ast \psi_{\epsilon}) \nabla\phi_R
    $$
and
    $$(i\partial_t + \Delta)v
    = \big( (i\partial_t + \Delta)u \ast \psi_{\epsilon} \big)\phi_R
        + (u \ast \psi_{\epsilon}) (i\partial_t + \Delta) \phi_R
        + 2(\nabla u \ast \psi_{\epsilon}) \cdot \nabla \phi_R.$$
Since $v$ is supported in $\{(x,t)\in\mathbb{R}^{n+1}: ct+|x|^2 \leq \varepsilon \}$
and we are assuming $u\in H^1_t \cap H^2_x$, by letting $R\rightarrow\infty$, we get
\begin{align*}
    \beta \big\| e^{\beta(ct+|x|^2)}(u\ast\psi_\varepsilon) \big\|_{L^2_{x,t}}
        + \beta^{\frac{1}{2}} \big\| e^{\beta(ct +|x|^2)} &|\nabla u\ast\psi_\varepsilon|\big \|_{L^2_{x,t}}\\
    &\leq \big\| e^{\beta(ct +|x|^2)} (i\partial_t + \Delta)(u\ast\psi_\varepsilon) \big\|_{L^2_{x,t}}.
    \end{align*}
Again by letting $\varepsilon\rightarrow0$,
\begin{equation*}
    \beta \big\| e^{\beta(ct+|x|^2)}u \big\|_{L^2_{x,t}}
        + \beta^{\frac{1}{2}} \big\| e^{\beta(ct +|x|^2)} |\nabla u|\big \|_{L^2_{x,t}}
    \leq \big\| e^{\beta(ct +|x|^2)} (i\partial_t + \Delta)u \big\|_{L^2_{x,t}}.
    \end{equation*}
Hence it follows that
    $$
    \begin{aligned}
    \beta \big\| e^{\beta(ct+|x|^2)} u &\big\|_{L^2_{x,t}(S)}
        + \beta^{\frac{1}{2}}\big\| e^{\beta(ct +|x|^2)} |\nabla u| \big\|_{L^2_{x,t}(S)} \\
    \leq &\big\| e^{\beta(ct +|x|^2)} (i\partial_t + \Delta) u \big\|_{L^2_{x,t}(S)}
        + \big\| e^{\beta(ct +|x|^2)} (i\partial_t + \Delta) u \big\|_{L^2_{x,t}(\mathbb{R}^{n+1}\setminus S)}.
    \end{aligned}
    $$
By \eqref{Sch}, we see that
    $$
    \begin{aligned}
    \big\| e^{\beta(ct +|x|^2)} (i\partial_t + \Delta) u \big\|_{L^2_{x,t}(S)}
    &\leq \|V\|_{L^{\infty}} \big\| e^{\beta(ct +|x|^2)} u \big\|_{L^2_{x,t}(S)}\\
        &+\big\||W|\big\|_{L^{\infty}} \big\| e^{\beta(ct +|x|^2)} |\nabla u| \big\|_{L^2_{x,t}(S)}.
    \end{aligned}
    $$
Hence if we choose $\beta$ large enough so that $\|V\|_{L^{\infty}}\leq\beta/2$ and $\big\||W|\big\|_{L^{\infty}}\leq\beta^{1/2}/2$,
we get
 $$
    \begin{aligned}
    \beta \big\| e^{\beta(ct+|x|^2)} u \big\|_{L^2_{x,t}(S)}
        + \beta^{\frac{1}{2}}&\big\| e^{\beta(ct +|x|^2)}|\nabla u| \big\|_{L^2_{x,t}(S)} \\
    &\leq2\big\| e^{\beta(ct +|x|^2)} (i\partial_t + \Delta) u \big\|_{L^2_{x,t}(\mathbb{R}^{n+1}\setminus S)}.
    \end{aligned}
    $$
Since $u$ vanishes in $\{(x,t)\in\mathbb{R}^{n+1}: ct+|x|^2 >0\}$ and $u\in H^1_t \cap H^2_x$,
we also see that
    $$
    \begin{aligned}
    \big\| e^{\beta(ct +|x|^2)} (i\partial_t + \Delta) u \big\|_{L^2(\mathbb{R}^{n+1}\setminus S)}
    &= \big\| e^{\beta(ct +|x|^2)} (i\partial_t + \Delta) u \big\|_{L^2 ( \{ (x,t): ct+|x|^2 \leq-1 \})}\\
    &\leq e^{-\beta} \big\| (i\partial_t + \Delta) u \big\|_{L^2 ( \{ (x,t): ct+|x|^2\leq-1 \} )} \\
    &\leq C e^{-\beta} .
    \end{aligned}
    $$
Hence, we get
    $$
    \beta \big\| e^{\beta(ct+|x|^2+1)} u \big\|_{L^2(S)}
        + \beta^{\frac{1}{2}} \big\| e^{\beta(ct +|x|^2+1)} |\nabla u| \big\|_{L^2(S)}
    \leq 2C.
    $$
Since $ct+|x|^2+1>0$ for $(x,t)\in S$,
    $$
    \beta \| u \|_{L^2(S)} + \beta^{\frac{1}{2}} \big\| |\nabla u|\big \|_{L^2(S)}\leq 2C.
    $$
By letting $\beta\rightarrow\infty$ we now conclude that $u=0$ on $S$.
This completes the proof.

\

\noindent\textbf{Acknowledgments.}
The authors thank the anonymous referees for many valuable suggestions which improve our presentation a great deal.




\begin{thebibliography}{9}

\bibitem{BFGRV} J. A. Barcelo, L. Fanelli, S. Gutierrez, A. Ruiz, M. C. Vilela,
\textit{Hardy uncertainty principle and unique continuation properties of covariant Schrodinger flows},
J. Funct. Anal. 264 (2013), 2386-2415.

\bibitem{C} T. Carleman, \emph{Sur un probl\`{e}me d'unicit\'{e} pour les syst\`{e}mes d'\'{e}quations
aux deriv\'{e}es partielles \`{a} deux variables ind\'{e}pendantes},
Ark. Mat., Astr. Fys. 26 (1939), 1-9.

\bibitem{CS} S. Chanillo and E. Sawyer, \emph{Unique continuation for $\Delta+v$ and the C. Fefferman-Phong class},
Trans. Amer. Math. Soc. 318 (1990), 275-300.

\bibitem{CR} F. Chiarenza and A. Ruiz, \emph{Uniform $L^2$-weighted Sobolev inequalities}, Proc. Amer. Math. Soc.
112 (1991), 53-64.

\bibitem{DS} H. Dong and W. Staubach, \textit{Unique continuation for the Schrodinger equation with gradient vector potentials}.
Proc. Amer. Math. Soc. 135 (2007), 2141-2149.

\bibitem{EKPV} L. Escauriaza, C. E. Kenig, G. Ponce and L. Vega, \textit{On uniqueness properties of solutions of Schr\"odinger equations},
Comm. Partial Differential Equations. 31 (2006), 1811-1823.

\bibitem{EKPV2} L. Escauriaza, C. E. Kenig, G. Ponce and L. Vega, \textit{Uniqueness properties of solutions to Schr\"odinger equations}, 
Bull. Amer. Math. Soc. (N.S.) 49 (2012), 415-442.

\bibitem{IK} A. D. Ionescu and C. E. Kenig, \emph{$L^p$ Carleman inequalities and uniqueness of solutions of
nonlinear Schr\"odinger equations}, Acta Math. 193 (2004), 193-239.

\bibitem{IK2} A. D. Ionescu and C. E. Kenig, \textit{Uniqueness properties of solutions of Schrodinger equations},
J. Funct. Anal. 232 (2006), 90-136.

\bibitem{I} V. Isakov, \textit{Carleman type estimates in anisotropic case and applications},
J. Diff. Eqs. 105 (1993), 217-238

\bibitem{JK} D. Jerison and C. E. Kenig, \textit{Unique continuation and absence of positive eigenvalues for
Schr\"odinger operators}, Ann. of Math. 121 (1985), 463-494.

\bibitem{KPV} C. E. Kenig, G. Ponce and L. Vega, \emph{On unique continuation for nonlinear Schr\"odinger
equations}, Comm. Pure Appl. Math. 56 (2003), 1247-1262.

\bibitem{KRS} C. E. Kenig, A. Ruiz and C. D. Sogge, \textit{Uniform Sobolev inequalities and
unique continuation for second order constant coefficient differential operators},
Duke Math. J. 55 (1987), 329-347.

\bibitem{KS} C. E. Kenig and C. D. Sogge, \emph{A note on unique continuation for Schr\"odinger's operator},
Proc. Amer. Math. Soc. 103 (1988), 543-546.

\bibitem{LS} S. Lee and I. Seo, \emph{A note on unique continuation for the Schr\"odinger equation},
J. Math. Anal. Appl. 389 (2012), 461-468.

\bibitem{LW} G. Lu and T. H. Wolff, \textit{Unique continuation with weak type lower order terms},
 Potential Anal. 7 (1997), 985-1026.

\bibitem{S} I. Seo, \emph{Unique continuation for the Schr\"odinger equation with potentials in Wiener amalgam spaces}, Indiana Univ. Math. J. 60 (2011), 1203-1227.

\bibitem{S2} I. Seo, \emph{Carleman estimates for the Schr\"odinger operator and applications to unique continuation}, Commun. Pure Appl. Anal. 11 (2012), 1013-1036.

\bibitem{S3} I. Seo, \textit{Global unique continuation from a half space for the Schr\"odinger equation}, J. Funct. Anal. 266 (2014), 85-98.

\bibitem{S4} I. Seo, \textit{Unique continuation for fractional Schr\"odinger operators in three and higher dimensions},
Proc. Amer. Math. Soc. 143 (2015), 1661-1664.

\bibitem{S5} I. Seo, \textit{Carleman inequalities for fractional Laplacians and unique continuation},
Taiwanese J. Math. 19 (2015), 1533-1540.

\bibitem{S6} I. Seo, \textit{From resolvent estimates to unique continuation for the Schr\"odinger equation},
Trans. Amer. Math. Soc., 368 (2016), 8755-8784.

\bibitem{St} E. M. Stein, \textit{Appendix to ``unique continuation''}, Ann. of Math. 121 (1985), 489-494.

\bibitem{W} T. H. Wolff, \emph{Unique continuation for $|\Delta u|\leq V|\nabla u|$ and related problems},
Rev. Mat. lberoam. 6 (1990), 155-200.

\bibitem{W2} T. H. Wolff, \textit{A property of measures in $\mathbb{R}^n$ and an application to unique continuation},
 Geom. Funct. Anal. 2 (1992), 225-284.

\bibitem{Z} B. Y. Zhang, \textit{Unique continuation properties of the nonlinear Schr\"odinger equations},
Proc. Roy. Soc. Edinburgh. 127 (1997) 191-205.

\end{thebibliography}
\end{document}